\documentclass{aptpub}

\usepackage{latexsym,amssymb,amsfonts,amsmath}
\usepackage{mathrsfs}
\usepackage{xcolor}
\usepackage{eucal}
\usepackage[normalem]{ulem}
\usepackage[colorlinks,citecolor=blue,urlcolor=blue]{hyperref}

\newcommand{\Real}{\mathbb R}
\newcommand{\eps}{\varepsilon}

\newcommand{\one}[1]{\mathbf{1}_{\{#1\}}}

\renewcommand{\P}{\mathsf{P}}

\renewcommand{\E}{\mathsf{E}}

\newcommand{\overbar}[1]{\mkern 1.5mu\overline{\mkern-1.5mu#1\mkern-1.5mu}\mkern 1.5mu}

\newcommand{\ben}{\begin{enumerate}}
\newcommand{\een}{\end{enumerate}}

\def\l{\lambda}
\def\m{\mu}

\def\eqs{\begin{eqnarray*}}
\def\ens{\end{eqnarray*}}

\def\eqa{\begin{eqnarray}}
\def\ena{\end{eqnarray}}

\authornames{J.~Baker, P.~Chigansky, K.~Hamza and F.C.~Klebaner} 
\shorttitle{Persistence of small noise and random initial conditions} 

\begin{document}

\title{Persistence of small noise and random initial conditions} 

\authorone[Monash University]{J. Baker} 

\addressone{School of Mathematical Sciences,
Monash University, Monash, VIC 3800, Australia;
Email address: jeremy.baker@monash.edu} 

\authortwo[The Hebrew University of Jerusalem]{P. Chigansky} 

\addresstwo{Department of Statistics,
The Hebrew University,
Mount Scopus, Jerusalem 91905,
Israel; Email address: Pavel.Chigansky@mail.huji.ac.il} 

\authorthree[Monash University]{K. Hamza} 

\addressthree{School of Mathematical Sciences,
Monash University, Monash, VIC 3800, Australia;
Email address: kais.hamza@monash.edu} 

\authorfour[Monash University]{F.C.Klebaner} 

\addressfour{School of Mathematical Sciences,
Monash University, Monash, VIC 3800, Australia;
Email address: fima.klebaner@monash.edu} 

\begin{abstract}
The effect of small noise in a smooth dynamical system is negligible on any finite time interval.
Here we study situations when it persists on intervals increasing to infinity. Such asymptotic regime
occurs when the system starts from initial condition, sufficiently close to an unstable fixed point.
In this case, under appropriate scaling, the trajectory converges to solution of the unperturbed
system, started from a certain {\em random} initial condition. In this paper we consider the case
of one dimensional diffusions on the positive half line, which often arise as scaling limits in
population dynamics.
\end{abstract}



\section{Introduction}

In this paper we study a new approximation for the stochastic process, generated by a nonlinear dynamical system
started in the vicinity of its unstable fixed point. The processes we consider can be represented as deterministic dynamics
perturbed by small noise. The well known results of Kurtz \cite{Kurtz} in the context of Markov jump processes or that of
Freidlin and Wentzell \cite{FW} in the context of diffusions, state  that the effect of small noise is negligible on any finite time
interval. This is  known as the fluid limit approximation, which implies that under appropriate conditions,
the small noise limit of the stochastic process solves the appropriate ordinary differential equation.

We are interested in the situation when the stochastic process starts near an unstable fixed point, such as zero,
in which case the usual fluid approximation results in zero, and thus is uninformative. To give a more adequate approximation
we let intervals increase to infinity. This yields fluid approximation with a {\em random} initial condition. The randomness comes as the random variable $W$, itself being the limit of the scaled linearized stochastic system, now  as time goes to infinity. The new initial condition is given by the suitable transformation of $W$, which is derived from the nonlinear deterministic dynamics in the fluid approximation.
Such approximation with random initial condition was obtained in \cite{BCK} for some nonlinear Birth-Death processes and in \cite{CJK} for
a discrete time model of Polymerase Chain Reaction (PCR).

In this paper we consider the case of one dimensional diffusions.  Such processes arise, for example, as approximations to the Wright-Fisher
model from evolutionary biology. Recently heuristics for random initial conditions was given in \cite{ML}, arguing
for Gamma type distribution, that is, a sum of independent exponentials. Our theory yields the random Poisson sum of similar exponentials.

The relevant approximation philosophy can be traced back to the works of Kendall \cite{Kendall56} and
Whittle \cite{Whittle55} in the context of Markovian SIR epidemic process. It was observed that
such processes should behave much like linear branching near the unstable fixed point and then should look more and more
like solutions to the deterministic equations. However rigorous implementation of the Kendall--Whittle heuristics turns out to be a
nontrivial matter, which requires quite different and model-specific techniques.

The main difficulty while working with dynamical systems on increasing time intervals
is that Gr\"{o}nwall's inequality, being the most common tool in this kind of analysis, ceases to be effective.
A more efficient approximation can be constructed by means of a two-stage procedure.
The whole time interval is split into two parts, on which the solution of the perturbed system is approximated in different ways:
first it is coupled to the linearized stochastic dynamics and then to its extrapolation
along the deterministic flow, generated by the unperturbed system.

The key ingredient of the method is the coupling, tailored to the particular type of dynamics on the case-to-case basis.
For the density dependent Birth-Death processes as in \cite{BCK}, this was done by means of
an auxiliary collection of Poisson processes, properly synchronized with jumps of the original system. For the one-dimensional
diffusions this approach is no longer feasible and, instead, we realize the coupling using the Yamada-Watanabe
type approximation by a linear process, driven by the same Brownian motion.

\section{The main result}

Consider the stochastic differential equation (s.d.e.) on $\Real_+$
\begin{equation}\label{main}
dX^\eps_t =  f(X^\eps_t) dt + \sqrt{\eps \sigma(X^\eps_t)} dB_t,\quad t\ge 0
\end{equation}
where $B_t$ is standard Brownian motion, $\eps>0$ is a small parameter and $f:\Real_+\mapsto \Real$ and $\sigma:\Real_+\mapsto \Real_+$
are twice continuously differentiable functions with bounded second derivative.
We assume that both functions vanish at zero, $f(0)=\sigma(0)=0$, and have positive slopes $f'(0)>0$ and $\sigma'(0)>0$, which makes
zero an unstable fixed point of \eqref{main} as well as of the ordinary differential equation (o.d.e.), obtained by removing the stochastic perturbation in \eqref{main}:
\begin{equation}\label{detmain}
\frac{dx_t}{dt}=f(x_t),\quad t\ge 0.
\end{equation}
In addition diffusion coefficient $\sigma(\cdot)$ and its derivative are assumed to be bounded
and $f(\cdot)$ satisfies the following drift condition:
\begin{equation}\label{drift}
(y-x)\big(f(y)-f(x)\big)\le f'(0) (y-x)^2, \quad x,y\in \Real_+.
\end{equation}

Smoothness of the coefficients and the drift condition \eqref{drift} are sufficient for existence of the unique strong solution
of \eqref{main} for any initial point $X^\eps_0\in \Real_+$ (see, e.g., \cite{Fimabook}).
Similarly the deterministic equation \eqref{detmain} admits unique continuous solution subject to any $x_0\in \Real_+$.
Henceforth we denote by $\phi_t(x)$ the flow generated by \eqref{detmain}.

The theory of small random perturbations, e.g. Freidlin and Wentzell \cite{FW},
asserts that the effect of noise on any {\em fixed} time interval $[0,T]$ is negligible as $\varepsilon\to 0$.

\begin{thm}\label{Kurtz}
Let  $X^\varepsilon_t$  satisfy \eqref{main}  and  $X^\eps_0\xrightarrow[\eps\to 0]{\P} x_0\in \Real_+$, then for any $T$
$$
\sup_{t\le T}|X^\varepsilon_t-x_t|\xrightarrow[\eps\to 0]{\P} 0,
$$
where $x_t$ is the solution of \eqref{detmain} subject to the initial condition $x_0$.
\end{thm}
Since zero is a fixed point of the limit dynamics \eqref{detmain}, this theorem implies that the
solution of \eqref{main}, started from a small positive initial condition $X^\eps_0=\eps>0$, converges to zero
on any fixed bounded interval
$$
\sup_{t\le T}\big|X_t^\eps\big| \xrightarrow[\eps\to 0]{\P} 0, \qquad \forall T\ge 0.
$$
On the other hand, since the fixed point is unstable and the initial condition is nonzero,
with positive probability, the trajectory $X^\eps_t$ is pushed out of the vicinity of the origin and, after sufficiently
large period of time, may reach a significant magnitude.  Being missed out by Theorem \ref{Kurtz}, this effect is captured
by the following result:

\begin{thm}\label{mainT} Let $X^\eps_t$ satisfy \eqref{main} subject to $X^\eps_0=\eps>0$ and define
$T_\eps := \dfrac 1 {f'(0)} \log \dfrac 1 \eps$.
Then for any $T>0$,
\begin{equation}\label{newapp}
\sup_{t\in [0,T]}\big| X^\eps_{T_{\eps} + t} -x_t \big| \xrightarrow[\eps \to 0]{\P} 0,
\end{equation}
where $x_t$ is the solution of \eqref{detmain} subject to the initial condition
$
x_0 = H(W)$. Function $H(x)$ is the limit of the scaled flow
\begin{equation}\label{Hxlim}
H(x) =  \lim_{t\to\infty} \phi_t\big(xe^{-f'(0) t}\big), \quad x\ge 0
\end{equation}
and the random variable $W$ is the a.s. martingale limit
$$
W :=\lim_{t\to\infty} e^{-f'(0)t} Y_t
$$
of the solution to the linearized system
\begin{equation}\label{Ylin}
Y_t =  1 + \int_0^t f'(0) Y_s ds + \int_0^t \sqrt{\sigma'(0) Y_s} dB_s.
\end{equation}

\end{thm}

\begin{rem}\

\medskip
\noindent
{\bf a)} Both random variable $W$ and function $H(\cdot)$, arising in the limit, admit explicit characterizations.
As recalled in Section \ref{sec-Fbd} below, $W$ has the compound Poisson distribution with rate $2a$ and exponentially
distributed jumps with mean $1/(2a)$. Function $H(\cdot)$ admits explicit expression \eqref{Gdef}, elaborated in Proposition
\ref{lemH}.

\medskip
\noindent
{\bf b)} Approximation \eqref{newapp} reveals that, when observed at a suitably increasing sequence of times, the trajectory of \eqref{main}
started near the unstable fixed point, converges to the solution of the same deterministic system \eqref{detmain},
as in  Theorem \ref{Kurtz}, but this time, with a random initial condition. Note that $\P(x_0=0)=\P(W=0)=e^{-2a}>0$ and hence
the limiting trajectory can be zero with positive probability. This corresponds to the event on which the process $X^\eps_t$
is absorbed at zero in a finite time. On the event $\{W>0\}$, the trajectories converge to a nontrivial curve,
whose initial point is random. This type of randomness was observed in biological models of sweeps
(see, e.g., \cite{ML}).
\end{rem}

\begin{example}[Wright-Fisher diffusion with selection]
The s.d.e.
$$
dX^\eps_t = aX^\eps_t (1 - X^\eps_t) dt + \sqrt{\eps} \sqrt{ X^\eps_t(1-X^\eps_t)} dB_t,
$$
fits the above framework with $f(x)= ax(1-x)$ and $\sigma(x)=x(1-x)$.
This equation describes evolution of an allele frequency in a population and is known as the
Wright-Fisher diffusion with selection coefficient $a$.
It has two absorbing states, 0 and 1, so that for $X^\eps_0\in [0,1]$ the trajectory is confined to the unit interval,
$
X^\eps_t\in [0,1]
$
for all $t\ge 0$. In particular, all the assumptions of Theorem \ref{mainT} are satisfied: $f(\cdot)$ and $\sigma(\cdot)$
can be defined arbitrarily outside the state space $[0,1]$ and hence their derivatives can be chosen bounded, etc.

Using the   expression for the deterministic flow
$$
\phi_{t}(x) = \frac{xe^{at}}{1-x+xe^{at}},
$$
it follows that
$$H(x) =  \lim_{t\to\infty} \phi_t(xe^{-at})=\frac{x}{1+x}.$$
Hence the random initial condition is given by
$$
x_0   =\frac W{W+1}.
$$
\hfill $\blacksquare$
\end{example}

\begin{example}[Balancing selection model]
The s.d.e.
$$
dX^\eps_t = aX^\eps_t (1 - X^\eps_t)(1 - 2 X^\eps_t) dt + \sqrt{\eps} \sqrt{ X^\eps_t(1-X^\eps_t)} dB_t,
$$
fits the above framework with $f(x)= ax(1-x)(1 - 2 x)$ and $\sigma(x)=x(1-x)$.
The fluid limit is given by the o.d.e.
\begin{equation*}
\frac{dx_t}{dt}=ax_t(1-x_t)(1-2x_t),\quad t\ge 0,
\end{equation*}
which generates the flow
\begin{equation*}
\phi_t(x) = \frac{1}{2}- \frac{1}{2}   \frac{1-2 x}{\sqrt{4 x (1-x)(e^{at}-1) +1}},\quad x \in (0,\tfrac 1 2).
\end{equation*}
It now follows that
$$H(x) =  \lim_{t\to\infty} \phi_t(xe^{-at})=\frac{1}{2}- \frac{1}{2} \frac{1}{\sqrt{4 x + 1 }}.$$
Hence the random initial condition in Theorem \ref{mainT} for this model is given by
$$
x_0 =  \frac{1}{2}- \frac{1}{2}   \frac{1}{\sqrt{4 W + 1 }}.
$$
\hfill $\blacksquare$
\end{example}

\section{Previous Results}

In this section we describe the progression of recent research into the question of approximation of populations started
near an unstable fixed point.

\subsection{Random time shift}

An important step in rigorous realization of ideas of  Kendall \cite{Kendall56} and
Whittle \cite{Whittle55}  was done in \cite{BHKK15}  for nonlinear Birth-Death processes.
The essence of this result is that the stochastic system is similar to the approximating deterministic one except for the {\it random} time shift.  To explain the setting, consider the ``bare bones" evolution model of Klebaner et al  \cite{KSVHJ2011}, in which a mutation appears
in an established population, and then the two subpopulations compete.
This is treated as a pure jump Markov process $Z^K(t)$ on $\mathbb{Z}_+^2$, whose first component counts
wild type individuals, initially around their carrying capacity, and the second
component counts mutant individuals.  The transition rates are as follows:
$$
\begin{array}{cclcl}
   z &\to& z + (1,0) \quad &\mbox{at rate}& \quad a_1z_1 \\
   z &\to& z + (-1,0) \quad &\mbox{at rate}& \quad z_1\big((z_1/K) + \gamma (z_2/K)\big) \\
   z &\to& z + (0,1) \quad &\mbox{at rate}& \quad a_2z_2 \\
   z &\to& z + (0,-1) \quad &\mbox{at rate}& \quad z_2\big(\gamma (z_1/K) + (z_2/K)\big)
   \end{array}
$$
where $K$ is a large parameter which controls carrying capacities of subpopulations and $a_1$, $a_2$ and $\gamma$ are positive
constants.

Initially, the first component $Z^K_1(0)$ has a value near its carrying capacity~$a_1K$ and no mutants are present in the population,
$Z^K_2 = 0$.  At time zero a number of mutant individuals are introduced.
The mutants and wild type individuals differ only through their birth rates $a_1$ and~$a_2$.
Each species has {\it per capita\/} death rate given by the density
of its own population, together with an additional component of $\gamma$ times the density
of individuals of the other species.  If $\gamma > 1$, members of the other species result
in a higher mortality rate than if they were of the same species; if $\gamma< 1$, they
result in a lower mortality rate, favouring the
possibility of coexistence.  If $a_2 < \gamma a_1$,
the mutants have negligible chance of survival, but, if $a_2 > \gamma a_1$, there is
a positive probability that the mutant strain will become established.  In this case,
if also $a_1 > \gamma a_2$, the two populations will eventually come to coexist.
%

Define the density process  $X_K(t)=Z_K(t)/K$ and let
$$
f(x)= \begin{pmatrix}
a_1x_1-x_1^2-\gamma x_1 x_2   \\
    a_2x_2-x_2^2-\gamma x_1 x_2
\end{pmatrix}.
$$
If $X_K(0)\to x(0)$ as $K\to\infty$, then by Theorem 3.1 in \cite{Kurtz}
$$
\sup_{0\le t\le T}\big|X_K(t) - x(t)\big| \xrightarrow[K\to\infty]{\P} 0,\quad T>0
$$
where $x(t)$ solves the o.d.e. \eqref{detmain} subject to initial condition $x(0)=(a_1,0)$, which is an unstable fixed point of the above deterministic dynamics. Hence the limit trajectory is constant and equals $x(0)$ on any finite interval $[0,T]$.

The heuristics for the approximation in \cite{BHKK15} is as follows.
Initially near $(a_1,0)$, the mutant component $Z^K_2(t)$ is approximated by a linear  Birth-Death process $Y(t)$
with per capita birth rate  $a_2$, death rate $\gamma a_1$ and the corresponding survival rate $\beta = a_2-\gamma a_1$,
which starts from $Y(0):=Z^K(0)$. Since
$Y(t)e^{-\beta t}$ is a martingale  with a non degenerate limit $W$,
\begin{equation}\label{stochapp}
X^K_2(t)=\frac 1 K Z^K_2(t)\approx \frac 1 K Y(t)\approx \frac 1 K  e^{\beta t}W=\frac{1}{K} e^{\beta( t+\beta^{-1}\log W)}.
\end{equation}
On the other hand, linearization of  dynamics \eqref{detmain} governed by $f(\cdot)$ near fixed point $(a_1,0)$ gives
$\dot{x}(t)\approx B x(t),$  where
$
B=\left(\begin{smallmatrix}
-a_1 &  -\gamma a_1 \\
0 & a_2-\gamma a_1
\end{smallmatrix}
\right)
$ is the Jacobian matrix. Hence
$
x(t) \approx e^{B t}x(0)
$
 and
\begin{equation}\label{detapp}
 x_2(t)\approx \frac{1}{K} e^{\beta t}Z^K_2(0)=\frac{1}{K} e^{\beta \big(t+\beta^{-1}\log Z^K_2(0)\big)}=
 \frac{1}{K} e^{\beta (t+\beta^{-1}\log \E W)},
\end{equation}
where we used the martingale property $Y(0)=\E Y(t) e^{-\beta t}=\E W$. Comparing the two approximations in
\eqref{stochapp} and \eqref{detapp}, we conclude that the stochastic path differs from the deterministic one by the random time shift
$\beta^{-1}(\log W- \log \E W).$ This heuristics is made precise by Theorem 1.2 in \cite{BHKK15}, which derives a non-asymptotic
approximation of solution to a certain class of stochastic systems, including the above ``bare bones" model as a special case,
by trajectories of the corresponding o.d.e. shifted by the random quantity as above.

\subsection{Random initial condition}

Fluid approximation on increasing time intervals also leads to differential equations with random
initial conditions. This phenomenon was recently studied in \cite{BCK} in the context of density dependent
populations. Let $Z^K_t$ be a continuous time Birth-Death process on $\mathbb{Z}_+$
with {\it per capita\/} birth rate $\lambda  - (\l-\m)g_1(z/K)$ and death rate $\mu   + (\lambda-\mu)g_2(z/K)$, $z\in \mathbb{Z}_+$,
where $\lambda>\mu\ge 0$ are fixed constants, $g(\cdot) = g_1(\cdot) + g_2(\cdot)$ is an increasing
function with $g(0)=0$ and $g(x_\infty)=1$ for some $x_\infty\in (0,\infty)$.
As before $K$ is a parameter, thought of as being large, that is representative of the  carrying capacity
of the population.

Process $Z^K_t$ admits the decomposition
\begin{equation}
\label{ZKt}
Z^K_t = Z^K_0 +  (\lambda  -\mu)\int_0^t Z^K_s\big(1 -  g (Z^K_s/K)\big)ds + M_t, \quad t\ge 0
\end{equation}
where $M_t$ is a martingale with predictable quadratic variation
$$
\langle M\rangle_t=\int_0^t Z^K_s\Big(\lambda +\mu  - (\l-\m)\big(g_2(Z^K_s/K)-g_1(Z^K_s/K)
\Big)ds.
$$
If we divide both asides of \eqref{ZKt} and define the density process $X^K_t:= \frac 1 K Z^K_t$ we get
$$
X^K_t = X^K_0 +  (\lambda  -\mu)\int_0^t X^K_s\big(1 - g (X^K_s)\big)ds + \frac 1 {\sqrt K}\widetilde M_t, \quad t\ge 0
$$
where the bracket of $\widetilde M_t:= \frac 1{\sqrt{K}}M_t$ depends only on $X^K_t$.
Therefore under appropriate technical conditions Theorem 3.1 from \cite{Kurtz}  implies that the density process $X^K_t$ converges
as $K\to\infty$ to the solution of the o.d.e.
\begin{equation}\label{ode}
   \frac d{dt} x_t =  (\lambda -\mu)x_t\big(1 -g(x_t)\big), \quad t\ge 0,
\end{equation}
subject to $x_0 = \lim_{K\to\infty} X^K_0$. Again, if the initial population size $Z^K_0$
is constant with respect to $K$, and hence the initial limit density is zero, $x_0=0$, the trivial limit
$x_t= 0$, $t\in [0,T]$ is obtained.

The main result in \cite{BCK} asserts that for any $T>0$ and $T_K:= \frac 1 {\lambda-\mu }\log K$
$$
\sup_{t\le T} \big|X^K_{t+T_K}-x_t\big|\xrightarrow[K\to\infty]{\P}0
$$
where $x_t$ solves \eqref{ode} subject to the random initial condition $x_0:= G^{-1}(\log W)$ with
\begin{equation} \label{GG}
G(x) =   \int_{0}^{x} \frac{g(u)}{u(1-g(u))}du + \log x, \quad x\in \Real_+.
\end{equation}
The random variable $W$ is the a.s. limit of the martingale
$$
W_t:= e^{-(\lambda-\mu)t}Y_t\xrightarrow[t\to\infty]{\P} W,
$$
where $Y_t$ is the continuous time Galton-Watson branching with constant per capita rates $\lambda$ and $\mu$, suitably defined
on the same probability space.

\subsection{An application to PCR}

Being motivated so far by theoretical considerations, fluid approximations on growing intervals
can also be of practical interest. One example is an application to
the Polymerase Chain Reaction (PCR) suggested in \cite{CJK}.
The model contains  the Michaelis-Menten constant $K$, large in terms of molecule numbers.
PCR typically starts from a very low concentration of initial DNA strands, which are not observable initially,  but become
observable after a number of iterations. This seems to be consistent with the theory, proposed in \cite{CJK}, which predicts that
DNA concentration raises to a measurable level after the number of repetitions of order $\log K$. Once the reaction becomes observable, the analytic approximation features random initial conditions, which can have practical implications.
Since this model is in discrete time, and uses discrete time techniques, we omit further details here.

\subsection{Contribution of this paper}

As mentioned in Introduction the key element of the program is the realization of coupling to the linear stochastic process.
While the broad term coupling is used, it actually  means different things in different situations.
For multidimensional Birth-Death processes in  \cite{BHKK15},  the coupling  is done by applying an abstract general result of Thorisson \cite{Thor}.   The non-linear and linear processes are coupled trajectories-wise on the set of full probability in the limit. This technique rests on the ability to evaluate total variation distance between  the non-linear and linear processes.
For the one-dimensional Birth-Death processes  in \cite{BCK}, the coupling is done by  constructing two linear processes with constant rates such that the non-linear process is sandwiched between these two. This is done in such a way that in the limit both linear processes and hence the non-linear one, converge to the same limit under appropriate scaling.

The results of the present paper are closer in spirit to those in \cite{BCK}.   Here coupling uses the same driving Brownian motion for the original non-linear process and its approximation with linear drift, which is the Feller branching diffusion \eqref{Ylin}. To show that these processes are close, we use a smooth approximation to the absolute value function, akin to the
Yamada-Watanabe approach to analysis of one-dimensional diffusions with non-smooth coefficients.

Our other contribution of a more conceptual flavour is the identification of the nonlinearity $H(\cdot)$ as the limit of the scaled flow
\eqref{Hxlim}, generated by the differential equation in the fluid approximation:
$$
H(x) =  \lim_{t\to\infty} \phi_t\big(xe^{-f'(0) t}\big).
$$
 In the one-dimensional case this limit admits the closed form expression  $H(x)=G^{-1}(\log x)$ with $G(\cdot)$ being defined in
\eqref{Gdef} below (c.f. \eqref{GG}).
We conjecture that this new charactrization remains valid in higher dimensional models, such as that considered in \cite{BHKK15},
and will prove to be useful in further exploration of the subject.

\section{Proof of Theorem \ref{mainT}}

Without loss of generality we fix the normalization $\sigma'(0)=1$ and denote $a:= f'(0)>0$.
The main step in the proof is to establish  convergence \eqref{newapp} at $t=0$, namely
\begin{equation}\label{t0}
 X^\eps_{T_\eps} \xrightarrow[\eps\to 0]{\P} H(W).
\end{equation}
The rest of the proof follows by a change of time.
Indeed, by letting $\widetilde X^\eps_t=X^\eps_{T_\eps+t}$,\\ and $\widetilde{B}_t = B_{T_\eps + t} - B_{T_\eps}$
we obtain from \eqref{main}
$$
\widetilde X^\eps_t =\widetilde X^\eps_{0} +\int_0^t f(\widetilde X^\eps_s)ds+
 \int_0^t \sqrt{ \eps \sigma(\widetilde X^\eps_s)}d\widetilde{B}_s,
$$
and the result follows from  \eqref{t0} by Theorem \ref{Kurtz}.

The proof of \eqref{t0} consists of a number of steps given as Propositions that follow. First we establish existence
of a nontrivial limiting function $H(\cdot)$, appearing in the random initial condition. Next we consider an auxiliary Feller branching
diffusion and its martingale limit $W$. Then we show convergence of processes on finite intervals under appropriate rescaling.
Finally all these ingredients are assembled together to construct the main approximation, which yields the statement of Theorem \ref{mainT}.

\subsection{The function $H(\cdot)$}

\begin{prop}\label{lemH}
The limit in \eqref{Hxlim} exists, uniformly on compacts, and is given by $H(x)=G^{-1}\big(\tfrac 1 a\log x\big)$ with
\begin{equation}\label{Gdef}
G(x) := \int_0^x \left(\frac 1{f(u)}-\frac 1{au}\right)du + \frac 1 a \log x.
\end{equation}
\end{prop}

\begin{proof}
Since $f$ is continuously differentiable, the flow $\phi_t(x)$ is differentiable in both variables and the derivative
$\phi_t'(x):= \partial_x \phi_t(x)$, $x>0$  satisfies
$$
\frac d {dt} \phi_t'(x) =  f'(\phi_t(x))\phi_t'(x), \quad t\ge 0
$$
subject to $\phi_0'(x)=1$.
Let $x^*$ be either the positive root of $f$, closest to the origin, or $x^* =\infty$, if $f(x)>0$ for all $x>0$.
Since the interval $(0,x^*)$ is invariant under the flow, we have $f(\phi_t(x))>0$ and therefore
$$
\begin{aligned}
\phi_t'(x) = & \exp\left(\int_0^t f'(\phi_s(x))ds\right)=
\exp\left(\int_0^t \frac{f'(\phi_s(x))}{f(\phi_s(x))}f(\phi_s(x))ds\right) =\nonumber\\
&
\exp\left(\int_0^t \frac{f'(\phi_s(x))}{f(\phi_s(x))}d \phi_s(x)\right)
=\frac{f(\phi_t(x))}{ f(x)} .
\end{aligned}
$$
Further, define $h_t =  \phi_t(xe^{-at})$. Then $h_t$  satisfies
\begin{align*}
h'_t = & \frac {d}{dt} \phi_t(xe^{-at}) = f(h_t) -a xe^{-at} \phi'_t(xe^{-at}) = \\
&
f(h_t) -a xe^{-at} f(\phi_t(xe^{-at}))/ f(xe^{-at}) = \\
&
f(h_t) \Big(1-a xe^{-at} / f(xe^{-at}) \Big).
\end{align*}
Let $t_0\ge 0$ be any point such that $xe^{-at_0}\in (0,x^*)$,
then rearranging and integrating  we get
$$
\int_{t_0}^t \frac{dh_s}{f(h_s)} = \int_{t_0}^t  \bigg(1-a \frac{xe^{-as} }{ f(xe^{-as})} \bigg)ds=
 \int_{xe^{-at}}^{xe^{-at_0}}  \bigg(\frac{1}{au} -\frac{1 }{ f(u)}\bigg)du, \quad t\ge t_0
$$

Since $f(\cdot)$ has bounded second derivative, for all sufficiently small $x>0$
\begin{align*}
\Big|\frac 1{ax}-\frac 1{f(x)}\Big| = & \Big|\frac{f(x)-ax}{ax f(x)}\Big| =
 \frac{1}{ax  f(x) } \Big|\int_0^x \int_0^u f''(v)dv du\Big|\le \\
&
\frac{ \|f''\|_\infty}{a  } \frac {x}{ f(x) } \le
\frac{ \|f''\|_\infty}{a  } \frac {x}{ax - x^2 \|f''\|_\infty}\le C
\end{align*}
with a constant $C$; in particular, the function $x\mapsto \frac 1{f(x)}-\frac 1{ax}$ is integrable at zero and we can define
$G(x)$ as in \eqref{Gdef} for $x\in (0,x^*)$.
This function is continuous and strictly increasing, since $G'(x)=1/f(x)>0$ for $x\in (0,x^*)$ and
\begin{equation}\label{Ght}
G(h_t)  - G(h_{t_0})=  \int_{xe^{-at}}^{xe^{-at_0}} \bigg(\frac{1}{au}-\frac{1 }{ f(u)} \bigg)du.
\end{equation}
For any fixed $c\in (0,x^*)$ we can write
$$
G(x) = \int_0^c \left(\frac 1{f(u)}-\frac 1{au}\right)du +\frac 1 a \log c +\int_c^x  \frac 1{f(u)} du.
$$
Hence when $x^*\in (0,\infty)$, we have  $\lim_{x\to x^*}G(x)=\infty$, since $f$ has bounded second derivative.
Note that drift condition \eqref{drift} implies $f(x)\le ax$ for all $x\ge 0$ and hence $\lim_{x\to x^*}G(x)=\infty$
also when $x^*=\infty$, that is when $f(x)>0$ for all $x>0$. Since we also have $\lim_{x\to 0}G(x)=-\infty$,
$G(\cdot)$ is a bijection from $(0,x^*)$ onto $\Real$ with continuous inverse. Therefore by \eqref{Ght} the limit
$H(x)=\lim_{t\to\infty}h_t(x)$ exists and satisfies
$$
G(H(x))   =  G(h_{t_0})+ \int_{0}^{xe^{-at_0}}  \bigg( \frac{1}{au}-\frac{1 }{ f(u)}\bigg)du,
$$
where the convergence is uniform over $x$ on compacts.  The claim follows, since the right hand side does not depend on the choice of $t_0$:
\begin{align*}
G(H(x))&  =  G(h_{t_0})+ \int_{0}^{xe^{-at_0}}  \bigg(\frac{1}{au}- \frac{1 }{ f(u)} \bigg)du =\\
&
\int_0^{h_{t_0}} \left(\frac 1{f(u)}-\frac 1 {au}\right)du + \frac 1 a\log h_{t_0}+ \int_{0}^{xe^{-at_0}}  \bigg(\frac{1}{au}- \frac{1 }{ f(u)} \bigg)du =\\
&
\int_{xe^{-at_0}}^{h_{t_0}}  \frac 1{f(u)} du
+\frac 1 a  \log x   -t_0=\frac 1 a \log x,
\end{align*}
where the last equality holds by the definition of $h_{t_0}$.
\end{proof}

\begin{rem} Function $H(x)$ satisfies a number of  properties.
\ben
\item
 It is a nontrivial solution of the o.d.e.
$$
H'(x)= \frac 1 {ax} f(H(x)), \quad x>0
$$
with $H(0)=0$. This can be seen directly from the explicit formula $H(x)=G^{-1}\big(\tfrac 1 a\log x\big)$.
%
%
\item $H(x)$ solves Schr\"{o}der's  functional equation
\begin{equation}\label{Sch}
H(x) = \phi_s \circ H (x e^{-as} ), \quad \forall \  x>0, \ s>0.
\end{equation}
Indeed by the semigroup property of the flow
\begin{align*}
\phi_t(xe^{-at}) = \phi_s \circ \phi_{t-s}\ (x e^{-as}e^{-a(t-s)}),
\end{align*}
and \eqref{Sch} is obtained by taking the limit $t\to\infty$ and using continuity.
\een
\end{rem}

\subsection{Feller's branching diffusion}\label{sec-Fbd}

The basic element of the approximation is Feller's branching diffusion
\begin{equation}\label{Feller}
Y_t=1+\int_0^taY_sds+\int_0^t\sqrt Y_s dB_s, \quad t\ge 0,
\end{equation}
driven by the same Brownian motion as in \eqref{main}.
The rescaled process $W_t:=e^{-at} Y_t$ is a nonnegative martingale with a non degenerate almost sure limit
\begin{equation}\label{W}
W:= 1+\int_0^\infty e^{-as} \sqrt{Y_s}dB_s.
\end{equation}
An explicit expression is available for the Laplace transform of $Y_t$ (see, e.g., Lemma 5 page 28 in \cite{Pardoux}):
$$
\E e^{-\lambda Y_t} = \exp \left(-\frac{\lambda a e^{at}}{a+\frac 1 2\lambda (e^{at}-1)}\right), \quad \lambda>0.
$$
Therefore
$$
\E e^{-\lambda W}=\lim_{t\to\infty}\E e^{-\lambda W_t}=
\exp \left(-\frac{2 a\lambda    }{2a+\lambda    }\right) = \E \exp \bigg(-\lambda\sum_{j=0}^\Pi \tau_j\bigg)
$$
with independent random variables $\Pi\sim \mathrm{Poi}(2a)$ and $\tau_j\sim \mathrm{Exp}(2a)$. As we will see,
it is the random variable $W$, which emerges in the limit claimed in Theorem \ref{mainT}.

\subsection{Approximation on bounded intervals}

The following lemma shows that the solution of \eqref{main} converges, under appropriate scaling,
to the Feller branching diffusion \eqref{Feller} on bounded intervals.

\begin{lem} \label{L1}
Let $\overbar X^\eps_t : =\eps^{-1} X^\eps_t$, where $X^\eps_t$ is the solution of \eqref{main} subject to $X^\eps_0=\eps$. Then
$$
\overbar X^\eps_t \xrightarrow[\eps\to 0]{L^1} Y_t, \quad \forall\, t\ge 0,
$$
where $Y_t$ is the solution of \eqref{Feller}.
\end{lem}

\begin{proof}

The process $\overbar X^\eps_t$ satisfies
\begin{equation}\label{barX}
 \overbar X^\eps_t = 1+ \int_0^t \eps^{-1} f(\eps \overbar X^\eps_s) ds + \int_0^t \sqrt{\eps^{-1} \sigma(\eps \overbar X^\eps_s)} dB_s,
\end{equation}
First let us show that the moments of $\overbar X^\eps_t$ are bounded, uniformly in $\eps$ on any finite time interval.
By drift condition \eqref{drift}, $f(x)\le ax$ for all $x\ge 0$ and the standard localization of the stochastic integral
$$
\E \overbar X^\eps_t =  1+ \E\int_0^t \eps^{-1} f(\eps \overbar X^\eps_s) ds   \le 1+ \int_0^t a   \E \overbar X^\eps_s  ds,
$$
and, in turn, by Gr\"{o}nwall's inequality
\begin{equation}\label{mom1}
\E \overbar X^\eps_t\le e^{at}, \quad \forall t>0.
\end{equation}
Now define $W^\eps_t := e^{-at}\overbar{X}^\eps_t$:
\begin{equation}\label{Wet}
\begin{aligned}
W^\eps_t = &
1 + \int_0^t e^{-as}\Big(\eps^{-1} f(\eps \overbar X^\eps_s)-a \overbar{X}^\eps_s\Big) ds + \int_0^t e^{-as}\sqrt{\eps^{-1} \sigma(\eps \overbar X^\eps_s)} dB_s \le \\
&
1  + \int_0^t e^{-as}\sqrt{\eps^{-1} \sigma(\eps \overbar X^\eps_s)} dB_s
\end{aligned}
\end{equation}
where the inequality holds since $\eps^{-1} f(\eps x)\le ax$. Since $W^\eps_t$ is positive,
$$
\E (W^\eps_t)^2 \le 1  + \|\sigma'\|_\infty\int_0^t e^{-2 as}  \E   \overbar X^\eps_s  ds\le 1+\frac {\|\sigma'\|_\infty} a,
$$
where we used \eqref{mom1}. Consequently
\begin{equation}\label{m2}
\E(\overbar X^\eps_t)^2\le \Big(1+\frac{\|\sigma'\|_\infty}{a}\Big)e^{2at}=:m_2(t).
\end{equation}

Further, for any $\eps\in (0,1)$ let
$$
\psi_\eps(u) :=\frac{1}{u\log(1/\sqrt\eps)}\one{\eps\le u\le \sqrt{\eps}}
$$
and define
$$
h_\eps(x) :=\int_0^{|x|}\int_0^y\psi_\eps(u)dudy, \quad x\in \Real.
$$
Obviously,
$$
h'_\eps(x) =\int_0^x\psi_\eps(u)du =  \begin{cases}
0 & 0<x\le \eps \\
\dfrac{\log (x/\eps)}{ \log(1/\sqrt\eps)} & \eps< x\le \sqrt{\eps} \\
1 & x> \sqrt{\eps}
\end{cases}
$$
and hence $|h'_\eps(x)|\le 1$. Since $h_\eps(x)$ is symmetric around zero,
$h'_\eps(x)\ge \one{x\ge \sqrt{\eps}}$ for $x>0$ and $h_\eps(0)=0$ we also have
$$
|x|\le h_\eps(x)+\sqrt{\eps}.
$$
Function $h_\eps(x)$ is a smooth approximation of $|x|$, used in the proof of Yamada-Watanabe theorem
and related applications (see, e.g., \cite{Gyongy}).

Note that from equations \eqref{Feller} and \eqref{barX} the difference
$Z^\eps_t= Y_t-\overbar X^\eps_t$ satisfies

$$
dZ^\eps_t =  a Z^\eps_t dt
-  \big(\eps^{-1} f(\eps \overbar X^\eps_t)-a  \overbar X^\eps_t\big) dt
+   \Big(\sqrt Y_t-\sqrt{\eps^{-1} \sigma(\eps \overbar X^\eps_t)}\Big) dB_t
$$
subject to $Z^\eps_0=0$.
Now we apply It\^o's formula to $h_{\eps}(Z_t^\eps)$ to get
\begin{equation}\label{Z}
|Z^\eps_t|\le \sqrt \eps + h_{\eps}(Z^\eps_t)=\sqrt \eps+\int_0^t I_{\eps} (s)ds+\frac{1}{2}\int_0^t J_{\eps} (s)ds+M_{ \eps }(t),
\end{equation}
where
\begin{align*}
I_{\eps} (t) := &
h'_\eps(Z^\eps_t) \Big( a Z^\eps_t   -  \big(\eps^{-1} f(\eps \overbar X^\eps_t)-a  \overbar X^\eps_t\big) \Big)\\
J_{ \eps }(t):= &  h''_{\eps}(Z^\eps_s) \Big(\sqrt Y_t-\sqrt{\eps^{-1} \sigma(\eps \overbar X^\eps_t)}\Big)^2 \\
M_{ \eps }(t):= & \int_0^th'_{\eps}(Z^\eps_s)\Big(\sqrt Y_s-\sqrt{\eps^{-1} \sigma(\eps \overbar X^\eps_s)}\Big)dB_s.
\end{align*}
Since $|h'_\eps(x)|\le 1$ and $|\eps^{-1} f(\eps x)-a  x|\le \eps\|f''\|_\infty  x^2$, the first term satisfies
$$
|I_{\eps} (s)|\le a|Z^\eps_s|+\eps \|f''\|_\infty(\overbar X^\eps_s)^2.
$$
Further, since $\sqrt x$ has global H\"older  exponent of $1/2$,
\begin{equation}\label{Hol2}
\big(\sqrt y-\sqrt{\eps^{-1} \sigma(\eps x)}\big)^2\le |y-\eps^{-1} \sigma(\eps x)|\le
|y-x|+\eps\|\sigma''\|_\infty x^2, \quad x,y\in \Real_+
\end{equation}
and therefore, using the estimate for $h''_\eps(x)$, we get
\begin{align*}
\big|J_{ \eps }(t)\big|\le\,
&
\big|h''_{\eps}(Z^\eps_s)\big| \Big(\sqrt Y_t-\sqrt{\eps^{-1} \sigma(\eps \overbar X^\eps_t)}\Big)^2 \le \\
&
\frac{1}{|Z^\eps_t|\log(1/\sqrt\eps)}
 \one{\eps\le |Z^\eps_t|\le \sqrt{\eps}}
\Big( \big|  Z^\eps_t \big|+  \eps\|\sigma''\|_\infty (\overbar X^\eps_t)^2 \Big) \le\\
&
\frac{1}{ \log(1/\sqrt\eps)}   \Big(1+\|\sigma''\|_\infty(\overbar X^\eps_t)^2 \Big).
\end{align*}
Similarly  by \eqref{Hol2}, the quadratic variation of $M_{\eps }$ is bounded by
\begin{align*}
\left<M_{\eps } \right>_t= &\,
\int_0^t \big( h'_{\eps}(Z^\eps_s)\big)^2\Big(\sqrt Y_s-\sqrt{\eps^{-1} \sigma(\eps \overbar X^\eps_s)}\Big)^2ds\le  \int_0^t |Z^\eps_s| ds
+\eps  \int_0^t (\overbar X^\eps_s)^2ds\\
&
\le  \int_0^t (Y_s+X^\eps_s) ds+\eps \|\sigma''\|_\infty \int_0^t (\overbar X^\eps_s)^2ds.
\end{align*}
Since the moments of $Y_t$ and $X^\eps_t$ are bounded on $[0,T]$ uniformly over $\eps$,  the latter implies
$$
\E\left<M_{\eps } \right>_t\le C_T<\infty,
$$
with a constant $C_T$ which depends only on $T$.
In particular, $M_{\eps }$ is a square integrable martingale with zero mean. Now taking expectation in \eqref{Z} and using the above bounds, we obtain
\begin{align*}
\E |Z^\eps_t|\le\, &  \sqrt \eps+a\int_0^t \E |Z^\eps_s|ds+\eps \|f''\|_\infty\int_0^t\E (\overbar X^\eps_s)^2 ds+
\\
&
\frac{1}{ \log(1/\sqrt\eps)}\int_0^t
   \Big(1+\|\sigma''\|_\infty\E (\overbar X^\eps_s)^2 \Big)
ds.
\end{align*}
Estimate \eqref{m2} and Gr\"{o}nwall's inequality imply $\E|Z^\eps_t|\to 0$ as $\eps\to 0$, which completes the proof.
\end{proof}

\subsection{The approximation}
The crux of our proof is an approximation of $X^\eps_t$ by means of a deterministic extrapolation of its trajectory
onwards from a certain suitably chosen time point.
To this end let us introduce deterministic and stochastic flows $\phi_{s,t}(x)$ and $\Phi_{s,t}(x)$ generated by
o.d.e. \eqref{detmain} and s.d.e. \eqref{main} respectively, i.e. the solutions of these equations at time $t$ that start at $x$ at
time $s$.
Further, let $\displaystyle t_c=\frac{c}{a} \log\frac 1 \eps$ with any constant $c\in (1/2,1)$ and  $t_1=T_\eps=\displaystyle\frac{1}{a} \log\frac 1 \eps$.
By these definitions
$X^\eps_{T_\eps} =\Phi_{t_c,t_1}(X^\eps_{t_c})$
and
\begin{equation}\label{decomp1}
 X^\eps_{T_\eps} =\big( \Phi_{t_c,t_1}(X^\eps_{t_c}) - \phi_{t_c, t_1}(X_{t_c}^\eps)\big)+ \phi_{t_c, t_1}(X_{t_c}^\eps).
\end{equation}
The convergence in \eqref{t0} holds once we check that the first term vanishes as $\eps\to 0$ and the second converges to $H(W)$
with the random variable $W$ from \eqref{W}.

\begin{lem}
$$
  \Phi_{t_c,t_1}(X^\eps_{t_c}) - \phi_{t_c, t_1}(X^\eps_{t_c})    \xrightarrow[\eps \to 0]{L^2} 0.
$$
\end{lem}

\begin{proof}
Let $\Phi^\eps_t :=  \Phi_{t_c,t_c+t }(X^\eps_{t_c})$ and $\phi_t := \phi_{t_c,t_c+t }(X^\eps_{t_c})$ for brevity
and define  $\delta^\eps_t = \Phi^\eps_t-\phi_t$.  Subtracting equations \eqref{main} and \eqref{detmain}  and applying the It\^o formula:
\begin{align*}
\E \big(\delta^\eps_t\big)^2  =\, &  \E \int_0^t 2\delta_s \big(f(\Phi^\eps_s) - f(\phi_s)\big)ds
+\int_0^t \eps\E  \sigma(\Phi^\eps_s) ds  \le \\
&
 \int_0^t 2a\E (\delta_s)^2  ds  +  \eps t \|\sigma\|_\infty
\end{align*}
where we used assumption \eqref{drift}.
By Gr\"{o}nwall's inequality
\begin{align*}
\E \Big(\Phi_{t_c,t_1}(X^\eps_{t_c}) - \phi_{t_c, t_1}(X^\eps_{t_c}) \Big)^2 =\;  &
\E \big(\delta^\eps_{t_1-t_c}\big)^2 \le \\
&
 C_1 \eps t_1 e^{2a(t_1-t_c)}\le
C_2  \eps^{2 c-1} \log \frac 1 \eps \xrightarrow[\eps\to 0]{} 0
\end{align*}
where the convergence holds by the choice $c\in (\frac 1 2, 1)$.
\end{proof}

The next lemma establishes convergence of the second term in \eqref{decomp1}:

\begin{lem}
$$
\phi_{t_c, t_1}(X_{t_c}^\eps)  \xrightarrow[\eps \to 0]{\P} H(W).
$$
\end{lem}
\begin{proof}
First we  show that   $W^\eps_{t_c}  =  e^{-at_c} \overbar{X}^\eps_{t_c}$
converges  in $L^2$ as $\eps\to 0$ to the limit $W$ from \eqref{W}, associated with the Feller branching diffusion.
Indeed, using representations \eqref{W} and \eqref{Wet}
\begin{align*}
W-W^\eps_{t_c}:= &
\int_{t_c}^\infty e^{-as} \sqrt{Y_s}dB_s - \int_0^{t_c} e^{-as}\Big(\eps^{-1} f(\eps \overbar X^\eps_s)-a \overbar{X}^\eps_s\Big) ds \\
&
+\int_0^{t_c} e^{-as} \Big(\sqrt{Y_s} -  \sqrt{\eps^{-1} \sigma(\eps \overbar X^\eps_s)}\Big) dB_s =: I_1(\eps)+I_2(\eps)+I_3(\eps).
\end{align*}
The first term converges to zero in $L^2$
$$
\E I_1(\eps)^2=\E\left(\int_{t_c}^\infty e^{-as}\sqrt Y_s dB_s\right)^2=\int_{t_c}^\infty e^{-2as}  \E Y_s ds\rightarrow 0,
$$
since for the Feller branching diffusion we have $\E Y_s=e^{as}$.
The second term converges to zero in $L^1$:
\begin{align*}
\E |I_2(\eps)| \le \int_0^{t_c} e^{-as}\eps \|f''\|_\infty \E (\overbar{X}^\eps_s)^2 ds\le
C_3 \eps  \int_0^{t_c} e^{as}   ds
\le C_4 \eps^{1-c}\xrightarrow[\eps\to 0]{}0.
\end{align*}

For the last term, we have
\begin{align*}
\E (I_3(\eps))^2=\, & \E\left(\int_0^{t_c}e^{-as}\Big(\sqrt{Y_s} -  \sqrt{\eps^{-1} \sigma(\eps \overbar X^\eps_s)}\Big)dB_s\right)^2= \\
&
\int_0^{t_c}e^{-2as}\E\Big(\sqrt{Y_s} -  \sqrt{\eps^{-1} \sigma(\eps \overbar X^\eps_s)}\Big)^2ds.
\end{align*}
Now,
$$\big(\sqrt{y}-\sqrt{\varepsilon^{-1}\sigma(\varepsilon x)}\big)^2\le |y-\varepsilon^{-1}\sigma(\varepsilon c)| \le |y-x| + |x-\varepsilon^{-1}\sigma(\varepsilon x)| \le |y-x| + \frac\varepsilon2\|\sigma''\|_\infty x^2$$
where the first inequality holds since the square root function is H\"older, and the last one is true because $\sigma'(0)=1$ was assumed (w.l.o.g.). Combining these two identities, we get

$$\E (I_3(\eps))^2 \le \int_0^{t_c}e^{-2as}\Big(\E|Y_s-\overbar{X}^\eps_s|+\frac\eps2 \|\sigma''\|_\infty \E(\overbar{X}^\eps_s)^2 \Big)ds   =:\int_0^{\infty}\one{s\le t_c}g_\eps(s)ds.$$

Since $\overbar X^\eps_s$ converges to $Y_s$ in $L^1$ by Lemma \ref{L1}, it follows that $g_\eps(s)\to 0$ as $\eps\to 0$
for any fixed $s\ge 0$. Further, exponential bounds on the moments \eqref{mom1} and \eqref{m2} imply
\begin{align*}
\one{s\le t_c} g_\eps(s)\le\, & C \one{s\le t_c}
e^{-2as}\Big( e^{as}+\eps e^{2as} \Big)
=
C \one{s\le t_c}   e^{-2as}\Big( e^{as}+e^{-(a/c) t_c} e^{ 2as} \Big)\le \\
&
C \one{s\le t_c}   e^{-2as}\Big( e^{as}+e^{-(a/c) s} e^{ 2as} \Big)\le
C     \Big( e^{-as}+  e^{ -(a/c)\, s} \Big),
\end{align*}
which is integrable on $\Real_+$. Hence by the dominated convergence the above integral converges to zero.
Thus $W^\eps_{t_c}$ converges to $W$ in $L^1$ as $\eps\to 0$.

Finally, by definition of $t_1$ and  $t_c$
$$
\phi_{t_c, t_1}(X_{t_c}^\eps)= \phi_{t_c, t_1}\big(W^\eps_{t_c} e^{-a(t_1-t_c)}\big)=\phi_{ t_1-t_c}\big(W^\eps_{t_c} e^{-a(t_1-t_c)}\big)\xrightarrow[\eps\to 0]{\P} H(W),
$$
where the limit holds by the uniform convergence from Proposition \ref{lemH}.
\end{proof}

\section*{Acknowledgments}
We are grateful to the referee for useful suggestions that improved the paper.
Research was supported by the Australian Research Council Grant DP150103588.
F.C. Klebaner is grateful to Professor Bohdan Maslowski of the Department of Probability and Mathematical Statistics, Charles University, Prague, for his hospitality.


\end{document}